\documentclass[reqno]{amsart}
\usepackage{amsmath,amsthm,amssymb,amscd}

\newtheorem{thm}{Theorem}[section]
\newtheorem{lem}{Lemma}
\newtheorem{cor}{Corollary}
\theoremstyle{definition}
\numberwithin{equation}{section}
\allowdisplaybreaks
\begin{document}

%
%
%
%
%
%
%
%
%

\title[$p$-adic quotient sets: Cubic forms]
 {$p$-adic quotient sets: Cubic forms}

\author{Deepa Antony}
\address{Department of Mathematics, Indian Institute of Technology Guwahati, Assam, India, PIN- 781039}
\email{deepa172123009@iitg.ac.in}

\author{Rupam Barman}
\address{Department of Mathematics, Indian Institute of Technology Guwahati, Assam, India, PIN- 781039}
\email{rupam@iitg.ac.in}

\date{October 23, 2021}



\dedicatory{}

\begin{abstract} 
For $A\subseteq \{1, 2, \ldots\}$, we consider $R(A)=\{a/b: a, b\in A\}$. It is an open problem to study the denseness of $R(A)$ in the $p$-adic numbers when $A$ is the set of nonzero values assumed by a cubic form. We study this problem for the cubic forms $ax^3+by^3$, 
where $a$ and $b$ are integers. We also prove that if $A$ is the set of nonzero values assumed by a non-degenerate, integral and primitive cubic form with more than 9 variables, then $R(A)$ is dense in $\mathbb{Q}_p$.
\end{abstract}

\maketitle
\section{Introduction and statement of results} 
Let $A$ be a subset of $\mathbb{N}=\{1, 2, 3, \ldots\}$. The set $R(A)=\{a/b:a,b\in A\}$ is called the ratio set or quotient set of $A$. 
Many authors have studied the denseness of ratio sets of different subsets of $\mathbb{N}$ in the positive real numbers. See for example \cite{real-1, real-2, real-3, real-4, real-5, real-6, real-7, real-8, real-9, real-10, real-11, real-12, real-13, real-14}.  
An analogous study has also been done for algebraic number fields, see for example \cite{algebraic-1, algebraic-2}. 
\par For a prime $p$, let $\mathbb{Q}_p$  denote the field of $p$-adic numbers. The study of denseness of ratio sets in $\mathbb{Q}_p$ is equally interesting. Such a study was initiated by Garcia and Luca \cite{garcia-luca} when they showed that the set of Fibonacci numbers is dense in $\mathbb{Q}_p$ for all primes $p$. Then Sanna \cite{Sanna} showed that the $k$-generalized Fibonacci numbers are dense in $\mathbb{Q}_p$. Later Garcia, Hong, Luca, Pinsker, Sanna, Schechter and Starr \cite{garciaetal} further broadened such investigation by considering different kinds of sets, for example, sum of squares and cubes of natural numbers. One of their results states that if $A=\{x^2+y^2: x, y\in \mathbb{Z}\}\setminus \{0\}$, then $R(A)$ is dense in $\mathbb{Q}_p$ if and only if $p\equiv 1\pmod 4$. They further posed some questions about denseness of ratio sets associated with quadratic forms and cubic forms.  They also asked a question on denseness of ratio sets associated with partitions of $\mathbb{N}$. For example, they asked whether there is a partition of $\mathbb{N}$ into two sets $A$ and $B$ such that $R(A)$ and $R(B)$ are dense in no $\mathbb{Q}_p$. Miska and Sanna \cite{miska-sanna} showed that the answer to this question is negative. In fact they proved more general result related to this question. In \cite{piotr}, Miska, Murru and C. Sanna studied the denseness of $R(A)$ in $\mathbb{Q}_p$ when $A$ is the image of $\mathbb{N}$ under a polynomial $f\in\mathbb{Z}[X]$.
The denseness of ratio set of nonzero values assumed by a quadratic form has been completely answered by Donnay, Garcia and Rouse \cite{Donnay}. They showed that for a binary quadratic form $Q$, $R(A)$ is dense in $\mathbb{Q}_p$ if and only if the discriminant of $Q$ is a nonzero square in $\mathbb{Q}_p$. They also proved that for a quadratic form in at least three variables, $R(A)$ is always dense in $\mathbb{Q}_p$. Later Miska \cite{miska} gave shorter proof for the same. 
\par 
The aim of this paper is to study this problem for the cubic forms.
A cubic form is a homogeneous polynomial
\begin{align*}
C(x_1,x_2, \ldots, x_r)=\sum_{i=1}^{r}\sum_{j=1}^{r}\sum_{k=1}^{r}a_{ijk}x_i x_j x_k 
\end{align*} 
of degree 3. We say that $C$ is integral if $a_{ijk} \in \mathbb{Z}$ for all $i, j, k$ and $C$ is primitive if there is no positive integer $d>1$ such that $d|a_{ijk}$ for all $i, j, k$. 
A form $C$ is said to be isotropic over a field $\mathbb{F}$ if there is a nonzero  vector $\overline{x}\in \mathbb{F}^r$ such that $C(\overline{x}) =0$. Otherwise  $C$ is said to be anisotropic over $\mathbb{F}$. The ratio set generated by a cubic form $C$ is 
\begin{align*}
R(C)=\{C(\overline{x})/C(\overline{y}): \overline{x}, \overline{y}\in \mathbb{Z}^r, C(\overline{y})\neq 0\}.
\end{align*}
In this article, we study denseness of $R(C)$ when $C$ is the cubic form $C(x,y)=ax^3+by^3$, where $a$ and $b$ are non-zero integers. To be specific, we prove the following two main results.
\begin{thm}\label{thm3}
	Let $C(x,y)=ax^3+by^3$ be primitive and integral. If $C$ is anisotropic modulo $p$, then $R(C)$ is not dense in $\mathbb{Q}_p$.
\end{thm}
For a prime $p$, let $\nu_p$ denote the $p$-adic valuation function.
\begin{thm}\label{thm4}
	Let $C(x,y)=ax^3+by^3$ be primitive and integral.
	\begin{enumerate}
		\item 	If $p\nmid ab$, then $R(C)$ is dense in $\mathbb{Q}_p$ if and only if $ba^{-1}$ is a cubic residue modulo $p^\alpha$, where $\alpha=1+\nu_p(3)$.
		\item If  $a=p^k\ell$ such that $p\nmid \ell$ and  $3|k$, then $R(C)$ is dense in $\mathbb{Q}_p$ if and only if $b^{-1}\ell$ is a cubic residue modulo $p^\alpha$, where $\alpha=1+\nu_p(3)$.
		\item If $a=p^k\ell,$ $p\nmid \ell$ and $3\nmid k$, then $R(C)$ is not dense in $\mathbb{Q}_p$. 
	\end{enumerate}
\end{thm}
Using Theorem \ref{thm4}, we have the following corollary.
\begin{cor}\label{cor-1}
	Let $C(x_1,x_2, \ldots, x_r)=a_1x_1^3+a_2x_2^3+\cdots +a_rx_r^3$ be primitive and integral. Suppose that $p\nmid a_i$ and $p^k|a_j$ for some $i, j$ from $1$ to $r, i\neq j$ and $3|k$. 
	Then $R(C)$ is dense in $\mathbb{Q}_p$ if $a_i^{-1}\ell$ is a cubic residue modulo $p^\alpha$, where $\alpha=1+\nu_p(3)$ and $a_j=p^k\ell,$ $p\nmid{\ell}$.
\end{cor}
\par 
Studying denseness of $R(C)$ when $C$ is any cubic form seems to be a difficult problem. We prove a general result on denseness of $R(C)$ when $C$ is a cubic form with more than 9 variables. 
Before we state our result, we recall some definitions. Two cubic forms $C_1$ and $C_2$ over $\mathbb{Q}_p$ are said to be equivalent if there is a non-singular linear transformation $T$ over $\mathbb{Q}_p$ such that $C_1(\overline{x})= C_2(T\overline{x})$. 
The order $\text{o}(C)$ of a cubic form $C$ is the smallest integer $m$ such that $C$ is equivalent to a form that contains only $m$ variables explicitly. A cubic form in $n$ variables is called non-degenerate if $\text{o}(C)=n$. 
In the following theorem we prove the denseness in $\mathbb{Q}_p$ of the ratio set of the values of a non-degenerate cubic form in more that 9 variables. 
\begin{thm}\label{thm7}
	Let $C$ be a non-degenerate, integral and primitive cubic form with more that 9 variables. Then $R(C)$ is dense in $\mathbb{Q}_p$.
\end{thm}
\section{Preliminaries}
For a prime number $p$, every nonzero rational number $r$ has a unique representation of the form $r= \pm p^k a/b$, where $k\in \mathbb{Z}, a, b \in \mathbb{N}$ and $\gcd(a,p)= \gcd(p,b)=\gcd(a,b)=1$. 
The $p$-adic valuation of such an $r$ is $\nu_p(r)=k$ and its $p$-adic absolute value is $\|r\|_p=p^{-k}$. By convention, $\nu_p(0)=\infty$ and $\|0\|_p=0$. The $p$-adic metric on $\mathbb{Q}$ is $d(x,y)=\|x-y\|_p$. 
The field $\mathbb{Q}_p$ of $p$-adic numbers is the completion of $\mathbb{Q}$ with respect to the $p$-adic metric. We denote by $\mathbb{Z}_p$ the ring of $p$-adic integers which is the set of values of $\mathbb{Q}_p$ with $p$-adic norm less than or equal to 1.
\par We next state a few results which will be used in the proof of our theorems.
\begin{lem}\label{lem1}\cite[Lemma 2.1]{garciaetal}
	If $S$ is dense in $\mathbb{Q}_p$, then for each finite value of the $p$-adic valuation, there is an element of $S$ with that valuation.
\end{lem}
\begin{lem}\label{lem2}\cite[Lemma 2.3]{garciaetal}
	Let $A\subset\mathbb{N}$.
	\begin{enumerate}
		\item 	If $A$ is $p$-adically dense in $\mathbb{N}$, then $R(A)$ is dense in $\mathbb{Q}_p$.
		\item 	If $R(A)$ is $p$-adically dense in $\mathbb{N}$, then $R(A)$ is dense in $\mathbb{Q}_p$.
	\end{enumerate}
\end{lem}
\begin{thm}\label{thm2}\cite[Corollary 1.3]{piotr}
	Let $f\colon\mathbb{Z}_p\rightarrow \mathbb{Q}_p$ be an analytic function with a simple zero in $\mathbb{Z}_p$, then  $R(f(\mathbb{Z}^+))$ is dense in $\mathbb{Q}_p$.
\end{thm}
We will need the following result to prove Theorem \ref{thm7}.
\begin{thm}\label{thm6}\cite[Theorem 2]{pleasants}
	Every cubic form $C$ over $K$  with $\text{o}(C)\geq 10$ has  a non-singular zero over $K$ where $K$ is any  field complete with respect to a discrete non-archimedean valuation.
\end{thm}
\section{Proof of the theorems}
\begin{proof}[Proof of Theorem \ref{thm3}]
	We claim that $\nu_p(C(x,y))$ is a multiple of 3 for all $x,y\in \mathbb{Z}$. If $C(x,y)\not\equiv 0\pmod{p}$, then $\nu_p(C(x,y))=0$. Suppose $C(x,y)\equiv0\pmod{p}$. 
	Then $(x,y)\equiv (0,0)\pmod{p}$ since $C$ is anisotropic; hence $x=mp^j$ and $y=np^k$ for some $j,k \geq 1$, $p\nmid m$ and $p\nmid n$. Without loss of generality, assume that $j\geq k$. Then,
	\begin{align*}
	\nu_p(C(x,y))
	&=\nu_p(am^3p^{3j}+bn^3p^{3k}) \\
	&=\nu_p(p^{3k}(am^3p^{3(j-k)}+bn^3))\\
	&=3k+\nu_p(C(mp^{(j-k)},n))=3k
	\end{align*}
	since $p\nmid n$ and $C$ is anisotropic. Thus by Lemma \ref{lem1}, $R(C)$ is not dense in  $\mathbb{Q}_p$.
\end{proof}
\begin{proof}[Proof of Theorem \ref{thm4}] We first prove part (1) of the theorem. Here $p\nmid ab$.
	Suppose that $p\neq 3$ and $ba^{-1}$ is a cubic residue modulo $p$. Then for a $y_0\in \mathbb{Z}$ with  $p \nmid y_0 $, there exists $x_0\in \mathbb{Z}$, 
	$p\nmid x_0$ such that $-ba^{-1}y_0^3\equiv x_0^3 \pmod{p}$. 
	Consider the polynomial $f(x)= x^3+ba^{-1} y_0^3$. Then $f(x_0)\equiv 0 \pmod {p}$ and $f^\prime(x_0)=3x_0^2 \not\equiv 0\pmod{p}$. Therefore, by Hensel's lemma \cite[(3.4)]{gouvea}, $f$ has a root $\alpha\in \mathbb{Z}_p$ which is a  simple root since $x_0 \not\equiv 0 \pmod{p}$. 
	By Theorem \ref{thm2}, $R(f(\mathbb{N}))$ is dense in $\mathbb{Q}_p$. Since $R(f(\mathbb{N}))\subset R(C)$, hence $R(C)$ is dense in $\mathbb{Q}_p$.
	Conversely, suppose that $ba^{-1}$ is not a cubic residue modulo $p$.
	If $C(x,y)\equiv 0 \pmod{p}$, that is, $ax^3+by^3\equiv 0 \pmod{p}$ then    
	\begin{align*}
	x^3\equiv -ba^{-1}y^3 \pmod {p}.
	\end{align*}
	The left side is either a cubic residue or divisible by $p$ and the right side is either a cubic non-residue or divisible by $p$. Hence $x\equiv y\equiv 0 \pmod{p}$. Therefore $C$ is anisotropic modulo $p$. 
	Hence $R(C)$ is not dense in $\mathbb{Q}_p$ by Theorem \ref{thm3}.
	\par 
	For $p=3$, consider $C'(x,y)=x^3+ba^{-1}y^3$. Suppose $ba^{-1}$ is a cubic residue modulo 9. Then by \cite[Lemma 3.5]{cassels}, $ba^{-1}$ is a cube in $\mathbb{Z}_3$, that is, there exists a non-zero $x_0\in \mathbb{Z}_3$ such that $x_0^3+ba^{-1}=0$. In particular, $x_0$ is a simple zero of the polynomial $f(x)=x^3+ba^{-1}$ in $\mathbb{Z}_3$. Therefore, by Theorem \ref{thm2}, $R(f(\mathbb{N}))$ is dense in $\mathbb{Q}_3$.  Since $R(f(\mathbb{N}))\subset R(C')$, $R(C)=R(C')$ is dense in $\mathbb{Q}_3$. Conversely, since $3\nmid ab$, $R(C)$ is dense in $\mathbb{Q}_3$ implies $ba^{-1}$ is a cube in $\mathbb{Z}_3$. Therefore, $ba^{-1}$ is a cubic residue modulo $3^k$ for each positive integer $k$, in particular, $ba^{-1}$ is a cubic residue modulo 9.    
	\par We next prove part (2) of the theorem. Here $a=p^k\ell$ such that $p\nmid \ell$ and $3|k$. We write $k=3k'$.
	Suppose that $b^{-1}\ell$ is not a cubic residue modulo $p^\alpha$. Let $C'(x,y)=b^{-1}\ell x^3+y^3$. By the first part of the theorem, $R(C')$ is not dense in $\mathbb{Q}_p$. We have $$b^{-1}C(x, y)=p^kb^{-1}\ell x^3+y^3=b^{-1}\ell(p^{k^\prime}x)^3+y^3=C'(p^{k^\prime}x, y).$$ 
	Since $R(C)=R(b^{-1}C)\subset R(C')$, $R(C)$ is not dense in $\mathbb{Q}_p$.
	\par 
	Conversely, suppose that $b^{-1}\ell$ is a cubic residue modulo $p^\alpha$. We have $b^{-1}C(x,y)=b^{-1}\ell p^{3k^\prime}x^3+y^3$. Using part (1) of the theorem, we have that $R(C')$ is dense in $\mathbb{Q}_p$, where  $C'(x,y)=b^{-1}\ell x^3+y^3$. Since 
	\begin{align*}
	\frac{C'(x,y)}{C'(z,w)}=\frac{p^{3k^\prime} C'(x,y)}{p^{3k^\prime}C'(z,w)}
	=\frac{C'(p^{k^\prime}x,p^{k^\prime}y)}{C'(p^{k^\prime}z,p^{k^\prime}w)}
	=\frac{C(x,p^{k^\prime}y)}{C(z,p^{k^\prime}w)},
	\end{align*}
	therefore, $R(C)$ is dense in $\mathbb{Q}_p$. This completes the proof of part (2) of the theorem.
	\par Finally, we prove part (3) of the theorem. Here $a=p^k\ell, p\nmid \ell$ and $3\nmid k$. Suppose that $R(C)$ is dense in $\mathbb{Q}_p$. Let $C'(x, y)=b^{-1}C(x, y)$. Then $R(C')$ is dense in $\mathbb{Q}_p$. 
	Choose  a cubic non-residue $n\pmod{p}$. There exist $x,y,z,w \in \mathbb{Z}$ not all multiples of $p$ such that
	\begin{align*}
	\left\Vert\frac{C'(x,y)}{C'(z,w)}-n\right\Vert_p < \frac{1}{p^k}.
	\end{align*}
	This yields
	\begin{align*} 
	&\left\Vert y^3-nw^3+p^kb^{-1}\ell (x^3-nz^3) \right\Vert_p \\
	&=\left\Vert C'(x,y)-nC'(z,w)\right\Vert_p\\
	&<\frac{\left\Vert C'(z,w)\right\Vert_p}{p^k}\\
	&\leq  \frac{1}{p^k}.
	\end{align*}
	If $p\nmid y$ or $p\nmid w$, then $y^3-nw^3 \not\equiv 0 \pmod{p}$ since $n$ is a cubic non-residue. Hence $\left\Vert C'(x,y)-nC'(z,w)\right\Vert_p=1$ which is a contradiction. 
	Therefore $p|y$ and $p|w$ which give $\nu_p(y^3-nw^3)=3m$ where $m$ is a positive integer. Since $p|y$ and $p|w$, we have either $p\nmid x$ or $p\nmid z$. This yields $\nu_p(x^3-nz^3)=0$. Since $3\nmid k$, therefore $C'(x, y)-nC'(z, w)$ is the sum of a $p$-adic integer with valuation a multiple of 3 and a $p$-adic integer with valuation not a multiple of 3.
	Hence $\left\Vert C'(x,y)-nC'(z,w)\right\Vert_p\geq p^{-k}$ which gives a contradiction. Thus, $R(C')$ is not dense in $\mathbb{Q}_p$. Hence $R(C)$ is not dense in $\mathbb{Q}_p$.
	This completes the proof of the theorem.
\end{proof}
\begin{proof}[Proof of Corollary \ref{cor-1}]
	It follows from Theorem \ref{thm4} that  $R(a_ix_i^3+a_jx_j^3)$ is dense in $\mathbb{Q}_p$. Since $R(a_ix_i^3+a_jx_j^3) \subset R(C)$, $R(C)$ is dense in $\mathbb{Q}_p$.
\end{proof}
We now prove Theorem \ref{thm7}. 
\begin{proof}[Proof of Theorem \ref{thm7}] Let $C$ be a non-degenerate, integral and primitive cubic form with more that 9 variables. Recall that a vector $\overline{x}\in \mathbb{Q}_p^r$ such that $C(\overline{x})=0$ is a 
	non-singular zero of $C$ if $\frac{\partial C}{\partial x_i}(\overline{x})\neq 0$ for some $i$ from 1 to $r$. In Theorem \ref{thm6}, taking $K= \mathbb{Q}_p$, we get a non-singular zero of $C$ over $\mathbb{Q}_p$. 
	By multiplying by an appropriate power of $p$, we get a non-singular zero $\overline{x_0}=(x_1,x_2, \ldots, x_r)$ of $C$ in $\mathbb{Z}_p^r$. We have
	\begin{align}
	C(\overline{x_0})=0, \frac{\partial C}{\partial x_i}(\overline{x_0})\neq 0 \label{eq1}
	\end{align}
	for some $i$ from 1 to $r$. Consider the polynomial $f(x)=C(x_1, \ldots,x, \ldots,x_r)$ in one variable $x$ obtained by replacing $x_i$ by $x$. By \eqref{eq1}, $f(x)$ has a simple root in $\mathbb{Z}_p$. 
	Therefore by Theorem \ref{thm2}, $R(f(\mathbb{N}))$ is dense in $\mathbb{Q}_p$. Since $R(f(\mathbb{N}))\subset R(C)$, $R(C)$ is dense in $\mathbb{Q}_p$. This completes the proof of the theorem.
\end{proof}
\section{Acknowledgements}
The authors are grateful to Piotr Miska for pointing out an error in an earlier draft of the article and for many helpful discussions. 

\end{document}